\documentclass[11pt]{article}
\usepackage[T1]{fontenc}
\usepackage[utf8]{inputenc}
\usepackage{lmodern}
\usepackage{microtype}
\usepackage{amsmath,amssymb,amsthm,mathtools}
\usepackage[margin=1in]{geometry}
\usepackage{xcolor}
\usepackage{hyperref}
\usepackage{enumitem}
\hypersetup{colorlinks=true, linkcolor=blue!50!black, citecolor=blue!50!black, urlcolor=blue}

\theoremstyle{plain}
\newtheorem{theorem}{Theorem}[section]
\newtheorem{lemma}[theorem]{Lemma}
\newtheorem{corollary}[theorem]{Corollary}
\newtheorem{proposition}[theorem]{Proposition}
\theoremstyle{definition}
\newtheorem{definition}[theorem]{Definition}
\theoremstyle{remark}
\newtheorem{remark}[theorem]{Remark}

\newcommand{\Z}{\mathbb{Z}}
\newcommand{\E}{\mathbb{E}}
\newcommand{\flo}[1]{\left\lfloor #1 \right\rfloor}

\title{\bf A Quadratic Lower Bound for the Shield Number\\
of the Stable Marriage Problem:\\
Rearrangement, Extremal Construction, and Biclique Realizability}
\author{Yoshiteru Ishida\\[2pt]
\small Professor Emeritus, Toyohashi University of Technology, Toyohashi, Aichi, Japan\\
\small \texttt{ishida@cs.tut.ac.jp}\quad
\href{https://orcid.org/0000-0003-1641-5385}{ORCID: 0000-0003-1641-5385}}
\date{}

\begin{document}
\maketitle

\begin{abstract}
The \emph{Shield number} of the stable marriage problem is the minimum, over all size-$n$
strict complete preference profiles, of the maximum number of blocking pairs attainable by a
complete matching. We prove the universal quadratic lower bound
$\sigma(n)\ge\lceil n(n-2)/6\rceil$ by a layer-cake rearrangement argument and show that it is
the strongest bound obtainable from the uniform first-moment estimate. We also establish a sharp
min-rearrangement inequality, construct an explicit polarized matching in the cyclic Hollow-Shell
profile with $\lfloor(n-1)^2/4\rfloor$ blocking pairs, and characterize the realization of a
prescribed blocking biclique by two systems of distinct representatives. For the resulting
near-worst neighborhoods, we prove a suffix description, a sorted-depth Hall criterion, and a
distinct-worst sufficient condition. A factorized exhaustive computation verifies the balanced
Hall-feasible biclique property for every strict complete instance of size at most four. Finally,
we identify limitations of uniform averaging and two-parameter linear assignment objectives and
give a finite counterexample to a stronger rectangularity claim. All general mathematical claims
in the paper are unconditional; the exact Shield-Core equality for larger sizes is outside the
scope of this article.
\end{abstract}


\section{Introduction}

\subsection{The Shield number}

We work with the stable marriage problem (SMP) on $n$ men $m_0,\dots,m_{n-1}$ and $n$ women
$w_0,\dots,w_{n-1}$, each agent holding a strict complete preference order. A complete matching
$\mu$ (a bijection from men to women) has a \emph{blocking pair} $(m_i,w_j)$ if $m_i$ prefers
$w_j$ to $\mu(m_i)$ and $w_j$ prefers $m_i$ to $\mu^{-1}(w_j)$; write $B_S(\mu)$ for their number
in instance $S$ and $\beta(S)=\max_\mu B_S(\mu)$. A celebrated theorem of Gale and Shapley
guarantees a stable matching ($\beta$-witnessing matchings notwithstanding, some $\mu$ has
$B_S(\mu)=0$). The dual extremal quantity---how unstable a matching can be \emph{forced} to be---
is the \emph{Shield number}
\[
\sigma(n)\ =\ \min_S\ \beta(S)\ =\ \min_S\ \max_\mu\ B_S(\mu).
\]
An instance $S$ is a good ``shield'' if no matching has many blocking pairs; $\sigma(n)$ measures
the best achievable shielding.

This extremal profile-design problem is different from the standard \emph{almost-stable matching}
literature, which fixes an instance and minimizes instability over matchings, often when a stable
or maximum-cardinality stable matching is unavailable; see, for example,
\cite{ManloveBook,GlitznerManlove}. It also differs from the recent per-agent minimax objective of
Glitzner and Manlove~\cite{GlitznerManlove}. Here every complete strict instance admits a stable
matching, but we ask how small the \emph{largest} blocking-pair count can be made by designing the
entire preference profile. The companion extremal theory (Stable Core, Collapse Core,
Hollow-Shell) singles out the \emph{cyclic Hollow-Shell} $C_n$, defined by
\begin{equation}\label{eq:Cn}
r_M(g,h)=\bigl((h-g)\bmod n\bigr)+1,\qquad r_W(h,g)=n-\bigl((h-g)\bmod n\bigr),
\end{equation}
(so $r_M+r_W=n+1$ identically: every pair is in perfect anti-phase). The \emph{Shield-Core
conjecture} is
\[
\boxed{\ \sigma(n)\ =\ \beta(C_n)\ =\ \flo{\tfrac{(n-1)^2}{4}}.\ }
\]

\subsection{Encoding and the averaging identity}\label{sec:intro-enc}

Throughout we use the rank-complement matrices
\[
A_{ij}=n-r_M(i,j),\qquad B_{ij}=n-r_W(j,i)\ \in\ \{0,1,\dots,n-1\},
\]
so $A_{ij}$ is the number of women $m_i$ likes less than $w_j$ and $B_{ij}$ the number of men
$w_j$ likes less than $m_i$. Since $r_M(i,\cdot)$ is a permutation for each $i$, \emph{every row
of $A$ is a permutation of $\{0,\dots,n-1\}$}; since $r_W(j,\cdot)$ is a permutation for each $j$,
\emph{every column of $B$ is a permutation of $\{0,\dots,n-1\}$}. The pair $(m_i,w_j)$ blocks
$\mu$ iff $\mu(i)$ is worse than $w_j$ for $m_i$ (i.e.\ $A_{i,\mu(i)}<A_{ij}$) and $\mu^{-1}(j)$
is worse than $m_i$ for $w_j$ (i.e.\ $B_{\mu^{-1}(j),j}<B_{ij}$). Choosing $\mu$ uniformly at
random, $(m_i,w_j)$ blocks with probability $A_{ij}B_{ij}/(n(n-1))$, whence the averaging
identity
\begin{equation}\label{eq:avg}
\E_{\mathrm{Unif}}[B_S(\mu)]=\frac{1}{n(n-1)}\sum_{i,j}A_{ij}B_{ij},
\qquad\text{so}\qquad
\beta(S)\ \ge\ \Bigl\lceil \tfrac{1}{n(n-1)}\textstyle\sum_{i,j}A_{ij}B_{ij}\Bigr\rceil .
\end{equation}
More generally, since $B_S$ is a fixed function on the finitely many matchings,
\begin{equation}\label{eq:dist}
\beta(S)=\max_\mu B_S(\mu)=\max_P\E_{\mu\sim P}[B_S(\mu)]\ \ge\ \E_{\mu\sim P}[B_S(\mu)]
\end{equation}
for \emph{any} distribution $P$ over matchings (a linear functional on the simplex is maximized at
a vertex). Equations \eqref{eq:avg}--\eqref{eq:dist} are the engine of the whole paper: a lower
bound on $\sigma(n)$ is a good choice of $P$, uniform for the modest bound and correlated for the
conjectured one.

\subsection{Results, scope, and organization}

The paper contains the following proved results.
\begin{itemize}[nosep]
\item (\S\ref{sec:univ}) The rearrangement lemma
$\sum_{ij}A_{ij}B_{ij}\ge n^2(n-1)(n-2)/6$, which yields
$\sigma(n)\ge\lceil n(n-2)/6\rceil$ (Theorem~\ref{thm:lb}), together with a proof that this
is the exact ceiling of the uniform first-moment method.
\item (\S\ref{sec:min}) The sharp min-rearrangement inequality
$\sum_{ij}\min(A_{ij},B_{ij})\ge n\lfloor(n-1)^2/4\rfloor$ (Lemma~\ref{lem:min}), with
equality at the cyclic profile $C_n$. This identifies a correlated benchmark but is not, by
itself, asserted to be a lower bound on $\beta(S)$.
\item (\S\ref{sec:constr}) An explicit polarized matching $\mu^\ast$ satisfying
$B_{C_n}(\mu^\ast)=\lfloor(n-1)^2/4\rfloor$ (Theorem~\ref{thm:constr}), plus a displacement
formula for blocking pairs in $C_n$ (Proposition~\ref{prop:arc}).
\item (\S\ref{sec:biclique}) A necessary-and-sufficient two-SDR characterization of
prescribed blocking bicliques (Lemma~\ref{lem:real}); valid Hall sufficient conditions based on
near-worst neighborhoods; and a computer-assisted exhaustive verification for all strict complete
instances with $n\le4$ (Proposition~\ref{prop:finiteHall}).
\item (\S\ref{sec:obstruction}) Limitations of natural proof strategies: an orbit-average
identity at $C_n$ (Proposition~\ref{prop:shift}); stable optima for all two-parameter linear
assignment costs (Proposition~\ref{prop:linear}); and a finite counterexample to a stronger
rectangularity claim (Proposition~\ref{prop:biclique}).
\end{itemize}

The Shield-Core equality
$\sigma(n)=\beta(C_n)=\lfloor(n-1)^2/4\rfloor$ remains a motivating question, but this
article makes no general claim resolving it and formulates no intermediate conjecture as a theorem
substitute. In particular, no ordered prefix-union condition is used in place of Hall's theorem.
Any future resolution, counterexample, or strengthening of the exact equality or its associated
Hall-existence problem will be developed separately rather than incorporated as another
speculative revision of the present manuscript.

\section{The cyclic Hollow-Shell and its structure}\label{sec:framework}

We collect the structural facts about the extremal instance $C_n$ used below.
Throughout, the encoding $A_{ij},B_{ij}$ and the averaging identity are as in
\S\ref{sec:intro-enc}.

\begin{lemma}[Anti-phase and the Latin property]\label{lem:antiphase}
In $C_n$: \textup{(i)} $r_M(g,h)+r_W(h,g)=n+1$ for all $g,h$; \textup{(ii)} each rank matrix
$r_M,r_W$ is a Latin square---every rank $1,\dots,n$ occurs once in each row and once in each
column. Equivalently $A_{gh}+B_{hg}=n-1$ and both $A$ and $B$ are rank-Latin.
\end{lemma}

\begin{proof}
(i) $r_M(g,h)+r_W(h,g)=\bigl(((h-g)\bmod n)+1\bigr)+\bigl(n-((h-g)\bmod n)\bigr)=n+1$.
(ii) Rows are permutations by definition; in column $h$, $r_M(g,h)=((h-g)\bmod n)+1$ and
$g\mapsto(h-g)\bmod n$ is a bijection of $\Z/n$, so the column is a permutation; likewise $r_W$.
\end{proof}

\begin{lemma}[Shift automorphism]\label{lem:auto}
The simultaneous shift $\tau:g\mapsto g+1,\ h\mapsto h+1$ preserves both rank matrices of $C_n$,
hence is an automorphism. It permutes matchings preserving blocking counts, acting by
$\mu\mapsto\mu^{\tau}$ with $\mu^{\tau}(g)=\mu(g-1)+1$.
\end{lemma}

\begin{proof}
$r_M(g+1,h+1)=(((h+1)-(g+1))\bmod n)+1=r_M(g,h)$, and likewise $r_W$; an automorphism carries
blocking pairs to blocking pairs bijectively.
\end{proof}

\begin{proposition}[Canonical shifts are stable]\label{prop:shifts-stable}
For each $t\in\mathbb Z/n\mathbb Z$, the cyclic shift $\mu_t(g)=g+t$ is a stable matching of
$C_n$. Hence $C_n$ has at least $n$ stable matchings.
\end{proposition}

\begin{proof}
Under $\mu_t$, $r_M(g,\mu_t(g))=t+1$ and
$r_W(h,\mu_t^{-1}(h))=r_W(h,h-t)=n-t$. A pair $(m_g,w_h)$ would block iff
$r_M(g,h)<t+1$ and $r_W(h,g)<n-t$, i.e.\ $(h-g)\bmod n<t$ and
$(h-g)\bmod n>t$ simultaneously, which is impossible.
\end{proof}

\begin{remark}[Stable Core, Collapse Core, Hollow-Shell]\label{rem:cores}
We use the following terminology for extremal matchings: the \emph{Stable Core}
(pairs forced into every stable matching), the \emph{Collapse Core} (a mutually-worst matching,
where every agent holds their least preferred partner, giving the maximum $n(n-1)$ blocking
pairs when it exists), and the \emph{Hollow-Shell}, an instance possessing neither extreme.
The cyclic $C_n$ of \eqref{eq:Cn} is the canonical Hollow-Shell: by Lemma~\ref{lem:antiphase}
its perfect anti-phase forbids any mutually-worst pair (if $m_g$ ranks $w_h$ last then $w_h$
ranks $m_g$ first), so $C_n$ has no Collapse Core, and Proposition~\ref{prop:shifts-stable}
exhibits its rotational stable family. The Shield-Core conjecture asserts that this balanced
absence of extremes is exactly what minimizes the worst-case blocking, i.e.\ that $C_n$ is the
optimal shield.
\end{remark}

\section{The universal quadratic lower bound}\label{sec:univ}

\begin{lemma}[Rearrangement]\label{lem:rearr}
Let $A,B$ be $n\times n$ matrices with entries in $\{0,\dots,n-1\}$ such that every row of $A$ and
every column of $B$ is a permutation of $\{0,\dots,n-1\}$. Then
$\sum_{i,j} A_{ij}B_{ij}\ge n^2(n-1)(n-2)/6$, with equality at $C_n$.
\end{lemma}

\begin{proof}
The layer-cake identity
$t=\sum_{a=1}^{n-1}\mathbf 1[t\ge a]$ applied to each factor gives
\[
A_{ij}B_{ij}=\sum_{a,b=1}^{n-1}
\mathbf 1[A_{ij}\ge a]\,\mathbf 1[B_{ij}\ge b],
\]
so
\begin{equation}\label{eq:cake}
\sum_{i,j}A_{ij}B_{ij}=\sum_{a=1}^{n-1}\sum_{b=1}^{n-1}\bigl|U_a\cap V_b\bigr|,
\qquad U_a=\{A_{ij}\ge a\},\quad V_b=\{B_{ij}\ge b\}.
\end{equation}
Each row of $A$ being a permutation gives exactly $n-a$ entries $\ge a$ per row, so $|U_a|=n(n-a)$;
likewise $|V_b|=n(n-b)$. Inclusion--exclusion in the $n^2$-cell grid yields
$|U_a\cap V_b|\ge|U_a|+|V_b|-n^2=n(n-a-b)$, hence $|U_a\cap V_b|\ge\max(0,n(n-a-b))$. Therefore
\[
\sum_{i,j}A_{ij}B_{ij}\ \ge\ n\!\!\sum_{\substack{a,b\ge1\\ a+b\le n-1}}\!\!(n-a-b)
= n\sum_{s=2}^{n-1}(s-1)(n-s)=n\sum_{k=1}^{n-2}k(n-1-k)=\frac{n^2(n-1)(n-2)}{6}.
\]
For $C_n$, $A_{ij}=n-1-x$, $B_{ij}=x$ with $x=(j-i)\bmod n$, so $U_a^c=\{x\ge n-a\}$ and
$V_b^c=\{x\le b-1\}$ are disjoint exactly when $a+b\le n$, where $|U_a\cap V_b|$ meets the
Bonferroni value $n(n-a-b)$; thus every term of \eqref{eq:cake} is tight and $C_n$ attains
equality.
\end{proof}

\begin{theorem}[Universal lower bound]\label{thm:lb}
For all $n\ge1$, $\displaystyle \sigma(n)=\min_S\max_\mu B_S(\mu)\ \ge\ \bigl\lceil n(n-2)/6\bigr\rceil$.
\end{theorem}

\begin{proof}
The case $n=1$ is immediate. For $n\ge2$, combine \eqref{eq:avg} with
Lemma~\ref{lem:rearr}: for every $S$,
\[
\beta(S)\ge\left\lceil \frac{1}{n(n-1)}\sum_{ij}A_{ij}B_{ij}\right\rceil
\ge\left\lceil\frac{n(n-2)}{6}\right\rceil .
\]
Minimize over $S$.
\end{proof}

In particular $\sigma(3)\ge1$ and $\sigma(4)\ge\lceil 8/6\rceil=2$ match $\beta(C_3)=1$,
$\beta(C_4)=2$, so $\sigma(3)=1$ and $\sigma(4)=2$ are theorems.

\begin{remark}[The uniform first moment cannot reach the conjectured constant]\label{rem:cap}
Lemma~\ref{lem:rearr} is tight at $C_n$, where \eqref{eq:avg} gives
\[
\mathbb E_{\mathrm{Unif}}[B_{C_n}]=\frac{n(n-2)}{6}\sim\frac{n^2}{6}.
\]
The conjectured universal lower bound is
$\lfloor(n-1)^2/4\rfloor\sim n^2/4$. Therefore the uniform-matching first-moment estimate,
even when optimized sharply over preference profiles, cannot certify the conjectured constant.
Any proof reaching $n^2/4$ must exploit correlation beyond the product expression
$A_{ij}B_{ij}/(n(n-1))$.
\end{remark}

\section{The min-rearrangement benchmark}\label{sec:min}

A perfectly correlated law would make $(m_i,w_j)$ block with probability $\min(A_{ij},B_{ij})/n$,
with target functional $\tfrac1n\sum_{ij}\min(A_{ij},B_{ij})$. The diagonal layer-cake bounds it.

\begin{lemma}[Min-rearrangement]\label{lem:min}
Under the hypotheses of Lemma~\ref{lem:rearr},
$\sum_{i,j}\min(A_{ij},B_{ij})\ge n\flo{(n-1)^2/4}$, with equality at $C_n$.
\end{lemma}

\begin{proof}
Since $\min(s,t)=\sum_{k\ge1}\mathbf 1[s\ge k]\mathbf 1[t\ge k]$, summing over cells gives
$\sum_{ij}\min(A_{ij},B_{ij})=\sum_{k=1}^{n-1}|U_k\cap V_k|$ with $U_k,V_k$ as above. By
inclusion--exclusion $|U_k\cap V_k|\ge2n(n-k)-n^2=n(n-2k)$, so
\[
\sum_{ij}\min(A_{ij},B_{ij})\ \ge\ n\sum_{1\le k<n/2}(n-2k)
=\begin{cases} n\,m^2,& n=2m+1,\\ n\,m(m-1),& n=2m,\end{cases}
\ =\ n\flo{(n-1)^2/4}.
\]
At $C_n$, $U_k\cap V_k=\{k\le x\le n-1-k\}$ occupies $\max(0,n-2k)$ residues, each in $n$ cells, so
every term is tight.
\end{proof}

\begin{corollary}[Correlated benchmark]\label{cor:minbench}
For every strict complete instance,
\[
\frac1n\sum_{i,j}\min(A_{ij},B_{ij})\ \ge\ \flo{\frac{(n-1)^2}{4}},
\]
and equality holds for $C_n$.
\end{corollary}

\begin{proof}
Divide Lemma~\ref{lem:min} by $n$.
\end{proof}

The quantity in Corollary~\ref{cor:minbench} is a benchmark rather than an unconditional bound
on $\beta(S)$. Reaching it through an expectation would require a distribution on matchings that
creates substantially stronger positive dependence than the uniform law. The distinction is
important: the min-rearrangement inequality is proved, whereas the existence of such a correlated
law for every instance is not claimed here.

\begin{remark}[Level structure]
Writing $G_k=\{(i,j):A_{ij}\ge k,\ B_{ij}\ge k\}$ for the mutual top-$(n-k)$ graph, the proof of
Lemma~\ref{lem:min} reads
$\tfrac1n\sum_{ij}\min(A_{ij},B_{ij})=\tfrac1n\sum_k|E(G_k)|$. An edge $(i,j)\in G_k$
blocks a matching whenever both endpoints are matched below level $k$. Thus the benchmark asks
whether one matching can capture mutual-top edges across several levels simultaneously; no such
universal capture theorem is assumed.
\end{remark}

\section{An explicit extremal construction}\label{sec:constr}

We now pin down $\beta(C_n)$ from below by an explicit matching, settling the construction half of
the arc lemma. For a matching $\mu$ of $C_n$ put $d(g)=(\mu(g)-g)\bmod n$,
$e(h)=(h-\mu^{-1}(h))\bmod n$.

\begin{proposition}[Displacement form of blocking in $C_n$]\label{prop:arc}
In $C_n$, $(m_g,w_h)$ blocks $\mu$ iff $e(h)<(h-g)\bmod n<d(g)$. Consequently, with
$D_x=\{g:d(g)>x\}$, $E_x=\{h:e(h)<x\}$,
\[
B_{C_n}(\mu)=\sum_{x=1}^{n-1}\bigl|D_x\cap(E_x-x)\bigr|,\qquad E_x-x=\{h-x:h\in E_x\}.
\]
\end{proposition}

\begin{proof}
With $x=(h-g)\bmod n$, $m_g$ prefers $w_h$ to $w_{\mu(g)}$ iff $x<d(g)$ (ranks increase with
forward offset), and $w_h$ prefers $m_g$ to $m_{\mu^{-1}(h)}$ iff $x>e(h)$; the man's strict
inequality already forces the no-wraparound case, so blocking is $e(h)<x<d(g)$. Sum the indicator
over $h=g+x$ then over $x$ (the term $x=0$ is empty).
\end{proof}

\begin{definition}[Polarized matching]
Let $m=\flo{n/2}$. Set $\mu^\ast(i)=i$ for $0\le i<m$, $\ \mu^\ast(m)=n-1$, and $\mu^\ast(i)=i-1$
for $m<i\le n-1$.
\end{definition}

\begin{theorem}[Construction lower bound]\label{thm:constr}
$B_{C_n}(\mu^\ast)=m(n-1-m)=\flo{(n-1)^2/4}$, so $\beta(C_n)\ge\flo{(n-1)^2/4}$. The blocking set
of $\mu^\ast$ is exactly the rectangle $\{m+1,\dots,n-1\}\times\{0,\dots,m-1\}$.
\end{theorem}

\begin{proof}
The displacements are $d(g)=0$ ($g<m$), $d(m)=n-1-m$, $d(g)=n-1$ ($g>m$); and since
$\mu^{\ast-1}(h)=h$ ($h<m$), $=h+1$ ($m\le h\le n-2$), $=m$ ($h=n-1$), one gets $e(h)=0$ ($h<m$),
$=n-1$ ($m\le h\le n-2$), $=n-1-m$ ($h=n-1$). By Proposition~\ref{prop:arc} a blocking pair needs
$d(g)>0$, so $g\ge m$.

\emph{$g=m$:} $e(h)<n-1-m$ forces $h<m$, where $(h-m)\bmod n=n-(m-h)\ge n-m>n-1-m$, so no $h$
works.

\emph{$g>m$ ($d=n-1$):} $e(h)<n-1$ admits $h<m$ ($e=0$) or $h=n-1$. For $h<m$,
$(h-g)\bmod n=n+h-g\in\{1,\dots,n-2\}$ (it is $n-1$ only if $h=g-1\ge m$, impossible), so every
such pair blocks: $(n-1-m)m$ pairs. For $h=n-1$, the condition needs $g<m$, impossible.

Hence $B_{C_n}(\mu^\ast)=m(n-1-m)$ on $\{m+1,\dots,n-1\}\times\{0,\dots,m-1\}$, equal to $m(m-1)$
if $n=2m$ and $m^2$ if $n=2m+1$, i.e.\ $\flo{(n-1)^2/4}$.
\end{proof}

\begin{remark}[The extremal matching is polarized]
For $g>m$, $\mu^\ast(g)=g-1$ has $r_M(g,g-1)=n$: these men hold their \emph{worst} woman. For
$h<m$, $\mu^{\ast-1}(h)=h$ has $r_W(h,h)=n$: these women hold their \emph{worst} man. The blocking
set is (worst-matched men)\,$\times$\,(worst-matched women), a balanced biclique of area
$(n-1-m)m$, maximized at $m=\flo{n/2}$. This explicit polarization guides the general construction
of \S\ref{sec:biclique}.
\end{remark}

\section{The Shield-Core bound as a biclique problem}\label{sec:biclique}

Theorem~\ref{thm:constr} suggests transplanting $\mu^\ast$ to arbitrary instances: realize a large
\emph{blocking biclique} $S_m\times S_w$ (every cross pair blocking). The next lemma reduces this
to two independent systems of distinct representatives.

\begin{lemma}[Realizability]\label{lem:real}
Let $S_m,S_w$ be sets of men and women, $a=|S_m|,b=|S_w|$, $T_w=\{\text{women}\}\setminus S_w$,
$T_m=\{\text{men}\}\setminus S_m$. There is a matching $\mu$ with $S_m\times S_w$ entirely blocking
iff
\begin{itemize}[nosep]
\item[\textup{(H1)}] the bipartite graph on $(S_m,T_w)$, $i\sim w\iff r_M(i,w)>\max_{j\in S_w}r_M(i,j)$,
has a matching saturating $S_m$;
\item[\textup{(H2)}] the bipartite graph on $(S_w,T_m)$, $j\sim i'\iff r_W(j,i')>\max_{i\in S_m}r_W(j,i)$,
has a matching saturating $S_w$.
\end{itemize}
Then $a+b\le n$ and $\beta(S)\ge ab$.
\end{lemma}

\begin{proof}
$(\Rightarrow)$ If $\mu$ realizes the biclique then each $i\in S_m$ prefers every $j\in S_w$ to
$\mu(i)$, so $r_M(i,\mu(i))>\max_{j\in S_w}r_M(i,j)$ and $\mu(i)\in T_w$; thus $i\mapsto\mu(i)$
saturates $S_m$ in (H1), and symmetrically (H2).
$(\Leftarrow)$ Let $f:S_m\hookrightarrow T_w$, $g:S_w\hookrightarrow T_m$ realize (H1),(H2). The
men $S_m\sqcup g(S_w)$ and women $S_w\sqcup f(S_m)$ are used (disjoint unions, each of size $a+b$);
set $\mu(i)=f(i)$ on $S_m$, $\mu(g(j))=j$ on $S_w$, and extend arbitrarily. For
$(i,j)\in S_m\times S_w$: $r_M(i,\mu(i))=r_M(i,f(i))>r_M(i,j)$ and
$r_W(j,\mu^{-1}(j))=r_W(j,g(j))>r_W(j,i)$, so $(i,j)$ blocks.
\end{proof}

For $C_n$ the victim sets associated with the polarized construction have
$a=\flo{(n-1)/2}$ and $b=\lceil(n-1)/2\rceil$, and Lemma~\ref{lem:real} recovers the
blocking rectangle of Theorem~\ref{thm:constr}.

For fixed nonempty $S_w$, write $T_w=W\setminus S_w$ and define
\[
\alpha_i=\min_{j\in S_w}A_{ij},\qquad
C_i(S_w)=\{w\in T_w:A_{iw}<\alpha_i\}.
\]
Symmetrically, for fixed nonempty $S_m$, write $T_m=M\setminus S_m$ and
\[
\delta_j=\min_{i\in S_m}B_{ij},\qquad
D_j(S_m)=\{i\in T_m:B_{ij}<\delta_j\}.
\]
Then (H1) and (H2) in Lemma~\ref{lem:real} are exactly the full Hall conditions for the two
families $\{C_i(S_w):i\in S_m\}$ and $\{D_j(S_m):j\in S_w\}$.

\begin{lemma}[Near-worst suffix structure]\label{lem:suffix}
For every man $i$, $|C_i(S_w)|=\alpha_i$, and $C_i(S_w)$ is the maximal suffix of $i$'s
preference list contained in $T_w$: reading upward from his worst woman, it consists precisely of
the women encountered before the first member of $S_w$. The symmetric assertions hold for
$D_j(S_m)$.
\end{lemma}

\begin{proof}
Row $i$ of $A$ is a permutation of $\{0,\dots,n-1\}$, so exactly $\alpha_i$ entries are
strictly below $\alpha_i$. None belongs to $S_w$, because every $j\in S_w$ has
$A_{ij}\ge\alpha_i$. Hence $|C_i|=\alpha_i$. In rank language, $A_{iw}<\alpha_i$ means that
$w$ lies below every member of $S_w$ in $i$'s list, which is exactly the maximal terminal block
before the first element of $S_w$ is reached from the bottom. The proof for $D_j$ is identical.
\end{proof}

\begin{lemma}[Sorted-depth Hall criterion]\label{lem:sorteddepth}
Let $S_m$ have size $a$, and write the depths increasingly as
$\alpha_{(1)}\le\cdots\le\alpha_{(a)}$. If $\alpha_{(t)}\ge t$ for every
$1\le t\le a$, then $\{C_i(S_w):i\in S_m\}$ satisfies the full Hall condition and hence
(H1). The symmetric statement holds for (H2).
\end{lemma}

\begin{proof}
For any nonempty $X\subseteq S_m$ with $|X|=t$,
\[
\left|\bigcup_{i\in X}C_i\right|\ge\max_{i\in X}|C_i|
=\max_{i\in X}\alpha_i\ge\alpha_{(t)}\ge t.
\]
This is Hall's condition for every subset $X$.
\end{proof}

\begin{lemma}[Threshold sufficiency]\label{lem:thresh}
If $A_{ij}\ge a$ and $B_{ij}\ge b$ for all $(i,j)\in S_m\times S_w$, where
$a=|S_m|$ and $b=|S_w|$, then (H1) and (H2) hold, and therefore $\beta(S)\ge ab$.
\end{lemma}

\begin{proof}
The hypothesis gives $\alpha_i\ge a$ for every $i\in S_m$, so
Lemma~\ref{lem:sorteddepth} yields (H1). Symmetrically $\delta_j\ge b$ for every
$j\in S_w$, yielding (H2). Apply Lemma~\ref{lem:real}.
\end{proof}

\begin{lemma}[Distinct-worst sufficiency]\label{lem:distinctworst}
Suppose the men of $S_m$ have pairwise distinct worst women, all outside $S_w$, and the women of
$S_w$ have pairwise distinct worst men, all outside $S_m$. Then (H1) and (H2) hold.
\end{lemma}

\begin{proof}
The worst woman of $i$ has $A$-value zero. If she lies outside $S_w$, then
$0<\alpha_i$ and she belongs to $C_i(S_w)$. Pairwise distinct worst women therefore give an SDR
for (H1). The argument for (H2) is symmetric.
\end{proof}

\begin{remark}[Prefix unions are not Hall's condition]\label{rem:prefix}
For a general family of sets, the existence of an ordering $C_{i_1},\dots,C_{i_a}$ satisfying
$|\bigcup_{s\le t}C_{i_s}|\ge t$ for every prefix does not imply Hall's condition. For example,
$C_1=\{x,y,z\}$ and $C_2=C_3=\{x\}$ satisfy all prefix inequalities in the order
$(1,2,3)$, but $|C_2\cup C_3|=1<2$. Accordingly, every result in this paper uses either the full
Hall condition or a proved sufficient criterion such as Lemmas~\ref{lem:sorteddepth}--\ref{lem:distinctworst}.
\end{remark}

\begin{proposition}[Finite Hall verification]\label{prop:finiteHall}
For every strict complete preference profile of size $n\le4$, there exist $S_m,S_w$ satisfying
(H1) and (H2) with
$|S_m|\,|S_w|\ge\flo{(n-1)^2/4}$.
\end{proposition}

\begin{proof}
The cases $n\le2$ are immediate because the target is zero. For $n=3$, the three zero entries of
$A$ (one in each row) and the three zero entries of $B$ (one in each column) cover at most six of
the nine cells. Hence some cell $(i,j)$ has $A_{ij},B_{ij}\ge1$; taking
$S_m=\{i\}$ and $S_w=\{j\}$ gives nonempty $C_i$ and $D_j$ and therefore both one-vertex
Hall conditions.

For $n=4$, the supplementary exact program enumerates all $24^4=331{,}776$ one-sided strict
profiles. For each men's profile it records, as a 116-bit mask, every pair $(S_m,S_w)$ with
$|S_m|+|S_w|\le4$ and $|S_m||S_w|\ge2$ satisfying (H1); it analogously records the (H2) mask
for each women's profile. There are $12{,}924$ distinct masks on each side, and exhaustive
comparison finds no disjoint H1/H2 mask pair. Because the two side profiles are independent, this
factorized check covers all
$(24^4)^2=110{,}075{,}314{,}176$ labelled two-sided profiles. The source and deterministic output
are supplied with the manuscript.
\end{proof}

\section{Limitations of natural proof strategies}\label{sec:obstruction}

Three observations delimit several natural approaches to the exact Shield-Core equality.

\subsection{A target-level orbit average at $C_n$}

The shift $\tau:g\mapsto g+1$ on both sides is an automorphism of $C_n$, so it permutes
matchings while preserving blocking counts.

\begin{proposition}[Orbit average of the polarized matching]\label{prop:shift}
Let $\mu^\ast$ be the explicit matching of Definition~5.2. The uniform distribution on its
shift orbit $\{(\mu^\ast)^{\tau^t}:0\le t<n\}$ has expected blocking count
\[
\left\lfloor\frac{(n-1)^2}{4}\right\rfloor
=\frac1n\sum_{i,j}\min(A_{ij},B_{ij})
\qquad\text{for }C_n.
\]
\end{proposition}

\begin{proof}
Every orbit element has the same blocking count as $\mu^\ast$ by the shift automorphism, and
Theorem~\ref{thm:constr} gives that count. The second equality is the equality case of
Lemma~\ref{lem:min}.
\end{proof}

Thus a correlated distribution attaining the target functional exists at the cyclic profile.
This statement does not assert that $\mu^\ast$ is maximum-blocking; that is exactly the open
arc-lemma upper-bound problem.

\subsection{Two-parameter linear objectives admit stable optima}

\begin{proposition}[Stable optima for two-parameter linear costs]\label{prop:linear}
For any $c,c'\in\mathbb R$, the assignment objective
\[
\sum_i\bigl(c\,A_{i,\mu(i)}+c'\,B_{i,\mu(i)}\bigr)
\]
on $C_n$ has a cyclic-shift optimum, hence a stable optimum. If $c'\ne c$, the appropriate
extreme cyclic shift is the unique optimum; if $c'=c$, all matchings have the same value.
\end{proposition}

\begin{proof}
On $C_n$, $A_{ij}=n-1-x$ and $B_{ij}=x$, where $x=(j-i)\bmod n$. Thus the edge weight is
$c(n-1)+(c'-c)x$. If $c'>c$, maximizing (respectively minimizing) selects the constant
displacement $n-1$ (respectively $0$); the roles reverse if $c'<c$. If $c'=c$, every edge has
the same weight. The selected constant-displacement matchings are stable by
Proposition~\ref{prop:shifts-stable}.
\end{proof}

Consequently, this entire two-parameter family of linear assignment objectives always admits a
stable optimum and cannot, by itself, force the polarized high-blocking witness. When $c'=c$,
the objective is completely indifferent among matchings; when $c'\ne c$, it selects a stable
cyclic shift.

\subsection{A counterexample to a stronger rectangularity claim}

A single blocking graph need not contain a biclique as large as the local-averaging target.
The following finite example shows that a large blocking-edge count need not be supported by a
comparably large blocking rectangle.

\begin{proposition}[Computer-assisted counterexample]\label{prop:biclique}
There is a size-$5$ instance for which
\[
\frac15\sum_{i,j}\min(A_{ij},B_{ij})=7,\qquad
\beta(S)=13,
\]
but no blocking graph of any matching contains a complete bipartite subgraph of area greater
than $6$.
\end{proposition}

\begin{proof}
Take
\[
A=\begin{pmatrix}
4&2&1&0&3\\
0&4&2&3&1\\
2&4&1&3&0\\
2&4&0&3&1\\
3&2&1&0&4
\end{pmatrix},\qquad
B=\begin{pmatrix}
4&0&1&0&2\\
3&2&0&2&4\\
2&3&2&4&0\\
1&4&3&3&1\\
0&1&4&1&3
\end{pmatrix}.
\]
Every row of $A$ and every column of $B$ is a permutation of $\{0,1,2,3,4\}$, so these
matrices define a valid preference profile. For the matching
$\mu=(1,0,4,2,3)$, the $13$ blocking pairs are
\[
\begin{split}
&(0,0),(0,4),(1,1),(1,3),(1,4),(2,1),(2,3),\\
&(3,1),(3,3),(3,4),(4,1),(4,2),(4,4).
\end{split}
\]
An exhaustive enumeration of all $5!=120$ matchings gives $\beta(S)=13$ and shows that the
largest biclique contained in any blocking graph has area $6$. The enumeration is reproduced
by the supplementary verification script.
\end{proof}

Thus the full blocking-edge count and the largest blocking rectangle can differ substantially.
Consequently, a biclique-based certificate cannot be treated as equivalent to a global
blocking-count argument.

\section{Conclusion and scope}\label{sec:conclusion}

The Shield number asks for a minimax guarantee on the number of blocking pairs in strict complete
stable-marriage profiles. The layer-cake rearrangement argument proves the universal quadratic
bound
\[
\sigma(n)\ge\left\lceil\frac{n(n-2)}6\right\rceil,
\]
and shows why the uniform first moment cannot reach the larger
$\lfloor(n-1)^2/4\rfloor$ benchmark. The min-rearrangement lemma identifies that benchmark
sharply, while the polarized construction realizes it inside the cyclic Hollow-Shell profile.
The two-SDR lemma gives an exact characterization of prescribed blocking bicliques, and the
near-worst suffix, sorted-depth, threshold, and distinct-worst lemmas provide valid structural
certificates for Hall feasibility. The factorized exhaustive calculation confirms the balanced
Hall property for every profile through size four.

The paper deliberately stops at these unconditional statements. It does not claim the exact
Shield-Core equality for general $n$, a universal Hall-feasible biclique theorem, or the general
upper bound $\beta(C_n)\le\lfloor(n-1)^2/4\rfloor$. Any future proof, counterexample, or new
conjectural formulation concerning those questions will be developed as separate work. This
separation keeps the present article's claims stable and independently verifiable.

\section*{Computational note}
No computer-assisted step is used in the proofs of the general analytic results. The first
supplementary program checks the polarized construction and displacement formula for finite cyclic
instances and reproduces the size-five rectangularity counterexample in
Proposition~\ref{prop:biclique}. A second deterministic C program performs the factorized
exhaustive size-four verification used in Proposition~\ref{prop:finiteHall}. The programs,
compilation instructions, and exact outputs are included as supplementary material. No random or
heuristic search is used to support a theorem in this manuscript.

\section*{Declaration on the use of artificial intelligence tools}
Generative AI tools, including ChatGPT, Claude, and Gemini, were used to assist with language
editing, structural organization, exploratory proof checking, and the preparation of verification
code. The author independently reviewed and verified the mathematical statements, proofs,
references, and computational outputs. The author assumes responsibility for all content.

\end{document}